\documentclass[12pt,reqno]{article}
\usepackage{amssymb}
\usepackage{amsmath}
\usepackage{amsthm}
\usepackage{amsfonts}
\usepackage{url}
\usepackage{hyperref}

\newcommand{\N}{{\mathbb N}}

\newcommand{\Z}{{\mathbb Z}}

\newcommand{\raisecomma}{\raisebox{2pt}{$,$}}

\newtheorem{theorem}{Theorem}
\newtheorem{corollary}{Corollary}
\newtheorem{lemma}{Lemma}

\newtheorem{remark}{Remark}

\newcommand{\seqnum}[1]{\href{http://oeis.org/#1}{\underline{#1}}}

\begin{document}
\bibliographystyle{plain}
\title{Some binomial sums involving absolute values}
\author{Richard P. Brent\\
Australian National University\\
Canberra, ACT 2600,
Australia\\
\href{mailto:binomial@rpbrent.com}{\tt binomial@rpbrent.com}
\and
Hideyuki Ohtsuka\\
Bunkyo University High School\\
1191-7, Kami, Ageo-City\\
Saitama Pref., 362-0001, Japan
\and
Judy-anne H. Osborn\\
The University of Newcastle\\
Callaghan, NSW 2308,
Australia
\and
Helmut Prodinger\\
Stellenbosch University\\
7602 Stellenbosch, South Africa
}

\date{}	

\maketitle
\thispagestyle{empty}                   

\begin{abstract}
We consider several families of binomial sum identities whose definition
involves the absolute value function. 
In particular, we consider centered double sums of the form
\[S_{\alpha,\beta}(n) := \sum_{k,\;\ell}\binom{2n}{n+k}\binom{2n}{n+\ell}
  |k^\alpha-\ell^\alpha|^\beta,\]
obtaining new results in the cases $\alpha = 1, 2$.
We show that there is a close connection between
these double sums in the case $\alpha=1$ and the single centered binomial
sums considered by Tuenter.
\end{abstract}

\pagebreak[3]
\section{Introduction}				\label{sec:intro}

The problem of finding a closed form
for the binomial sum
\begin{equation}				\label{eq:motivation}
\sum_{k,\;\ell}\binom{2n}{n+k}\binom{2n}{n+\ell}|k^2-\ell^2|
\end{equation}
arises in an application of the probabilistic
method to the Hadamard maximal determinant problem~\cite{rpb257}.
Because of the double-summation and the 
absolute value occurring in~\eqref{eq:motivation},
it is not obvious how to apply standard techniques~\cite{Gosper,AeqB,Wilf}.
A closed-form solution
\begin{equation}
2n^2\binom{2n}{n}^2		\label{eq:BO_rpb255}
\end{equation}
was proved by Brent and Osborn in~\cite{rpb255},
and simpler proofs were subsequently found~\cite{rpb260,Callan,Prodinger1}.
In this paper we consider
a wider class of binomial sums with the distinguishing feature
that an absolute value occurs in the summand.

Specifically, we consider certain $d$-fold binomial sums of the form
\begin{equation}	\label{eq:general_form}
S(n) := \sum_{k_1,\ldots,k_d}
	\prod_{i=1}^d \binom{2n}{n+k_i}\; |f(k_1,\ldots,k_d)|,
\end{equation}
where 
$f:\Z^d \rightarrow \Z$ is a homogeneous polynomial
and $|f|$ will be called the \emph{weight function}.
For example, a simple case
is $d=1$, $f(k) = k$. This case was considered by
Best~\cite{Best} in an application to Hadamard matrices.
The closed-form solution is
\[
\sum_k \binom{2n}{n+k}|k| = n\binom{2n}{n}.\]
A generalization
$f(k) = k^r$ (for a fixed $r \in \N$)
was considered by Tuenter~\cite{Tuenter1}, and shown to be expressible
using Dumont-Foata polynomials~\cite{DF}.
Tuen\-ter gave an interpretation in terms of the moments of the distance
to the origin in a symmetric Bernoulli random walk.  It is easy to see
that this interpretation generalizes:
$4^{-nd}S(n)$ is the expectation of $|f(k_1,\ldots,k_d)|$
if we start at the origin and take $2n$ random steps $\pm\frac{1}{2}$
in each of
$d$ dimensions, thus arriving at the point $(k_1,\ldots,k_d)\in\Z^d$
with probability 
\[
4^{-nd}\prod_{i=1}^d \binom{2n}{n+k_i}.
\]
A further generalization replaces 
$\binom{2n}{n+k_i}$ by $\binom{2n_i}{n_i+k_i}$, allowing the number of
random steps $(2n_i)$ in dimension~$i$ to depend on~$i$. With a suitable
modification to the definition of $S$, we could also drop the restriction to
an even number of steps in each dimension.\footnote{For
example, in the case $d=1$ we could consider $\sum_k \binom{n}{k}|f(n-2k)|$.}
We briefly 
consider such a generalization in \S\ref{sec:reducible}.

Tuenter's results for the case $d=1$ were generalized
by the first author~\cite{rpb259}.
In this paper we concentrate on the case $d=2$.  
Generalizations of some of the results
to arbitrary $d$ are known.
More specifically, the paper~\cite{rpb263} gives closed-form solutions for the
$d$-dimensional generalization of the sum~\eqref{eq:Srmndef} below
in the cases $\alpha,\beta\in\{1,2\}$.

There are many binomial coefficient identities in the literature, e.g.\
$500$ are given by Gould~\cite{Gould}.  Many such identities can be
proved via generating functions~\cite{GKP,Wilf} or
the Wilf-Zeilberger algorithm~\cite{AeqB}.
Nevertheless, we hope that the reader will find our results interesting, in
part because of the applications mentioned above, 
and also because it is a challenge
to generalize the results to higher values of~$d$.

A preliminary version of this paper, with some of the results
conjectural, was made available on arXiv~\cite{rpb260}. All the
conjectures have since been proved by 
Bostan, Lairez and Salvy~\cite{Bostan15},
Krattenthaler and Schneider~\cite{KS},
Brent, Krattenthaler and Warnaar \cite{rpb263},
and the present authors.

An outline of the paper follows.

In \S\ref{sec:reducible} we consider a special
class of double sums that can be
reduced to the single sums of~\cite{rpb259,Tuenter1}.

In \S\ref{sec:2fold} we consider a generalization of the
motivating case~\eqref{eq:motivation} described above:
$f(k,\ell) = (k^\alpha-\ell^\alpha)^\beta$.
In the case $\alpha=2$
we give recurrence relations that allow such sums to be evaluated
in closed form for any given positive integer~$\beta$.
The recurrence relations naturally split
into the cases where $\beta$ is even (easy)
and odd (more difficult).

Theorem~\ref{thm:3-fold} in
\S\ref{sec:triple} gives a closed form for an analogous triple sum.
In~\cite[Conjecture 2]{rpb260} 
a closed form for the analogous
quadruple sum was conjectured. This conjecture has now been
proved by Brent, Krattenthaler and Warnaar~\cite{rpb263}; in fact they
give a generalization to arbitrary positive integer~$d$.

In \S\ref{sec:further} we state several double sum
identities that were proved or conjectured by us~\cite{rpb260}.
The missing proofs have now been provided by 
Bostan, Lairez and Salvy~\cite{Bostan15}
and by Krattenthaler and Schneider~\cite{KS}.

\subsection*{Notation}

The set of all integers is $\Z$, and the set of non-negative integers
is $\N$,  

The binomial coefficient $\binom{n}{k}$ is defined to be zero if $k<0$ or
$k>n$ (and hence always if $n < 0$). 
Using this convention, we often avoid explicitly specifying upper and
lower limits on $k$ or excluding cases where $n < 0$.

In the definition of the weight function $|f|$, we always interpret
$0^0$ as $1$.

\section{Some double sums reducible to single sums}	\label{sec:reducible}

Tuenter~\cite{Tuenter1} considered the binomial sum
\begin{equation}	\label{eq:Sr_def}
S_\beta(n) := \sum_k \binom{2n}{n+k} |k|^\beta,
\end{equation}
and a generalization%
\footnote{It is a generalization because $S_\beta(n) = U_\beta(2n)$,
but $U_\beta(n)$ is well-defined for all $n\in\N$.
}
to
\begin{equation}		\label{eq:Ur_def}
U_\beta(n) := \sum_k \binom{n}{k}\left|\frac{n}{2}-k\right|^\beta
\end{equation}
was given by the first author~\cite{rpb259}.

Tuenter showed that
\begin{equation}		\label{eq:PQ_def}
S_{2\beta}(n) = Q_\beta(n)2^{2n-\beta}, \;\; 
  S_{2\beta+1}(n) = P_\beta(n)n\binom{2n}{n},
\end{equation}
where $P_\beta(n)$ and $Q_\beta(n)$ are polynomials of degree~$\beta$
with integer
coefficients, satisfying certain three-term recurrence relations, and
expressible in terms of Dumont-Foata polynomials~\cite{DF}.
Closed-form expressions for $S_\beta(n)$, $P_\beta(n)$, $Q_\beta(n)$
are known~\cite{rpb259}.

In this section we consider the double sum
\begin{equation}	\label{eq:Tr_def}
T_\beta(m,n) := \sum_{k,\;\ell}
 \binom{2m}{m+k}\binom{2n}{n+\ell}|k-\ell|^\beta
\end{equation}
and show that it can be expressed as a single sum of the
form~\eqref{eq:Sr_def}.

\begin{theorem}		\label{thm:reduce2to1}
For all $\beta, m, n \in \N$, we have
\[T_\beta(m,n) = S_\beta(m+n),\]
where $T_\beta$ is defined by~$\eqref{eq:Tr_def}$ and
$S_\beta$ is defined by~$\eqref{eq:Sr_def}$.
\end{theorem}

\begin{proof}
If $\beta=0$ then $T_0(m,n) = 2^{2(m+n)} = S_0(m+n)$.
Hence, we may assume that
$\beta>0$ (so $0^\beta = 0$).  Let $d = |k-\ell|$.
We split the sum~\eqref{eq:Tr_def} defining
$T_\beta(m,n)$ into three parts, corresponding to $k > \ell$, $k < \ell$, and
$k=\ell$.  The third part vanishes.
If $k>\ell$ then $d=k-\ell$ and $k=d+\ell$;
if $k<\ell$ then $d=\ell-k$ and $\ell=d+k$.
Thus, we get
\begin{align*}
T_\beta(m,n) &=
\sum_{d>0} \sum_\ell \binom{2m}{m\!+\!d\!+\!\ell}\binom{2n}{n\!+\!\ell}d^\beta
 +\sum_{d>0} \sum_k \binom{2m}{m\!+\!k}\binom{2n}{n\!+\!k\!+\!d}d^\beta\\
&= \sum_{d>0} d^\beta \sum_\ell \!\binom{2m}{m\!+\!d\!+\!\ell}
 \!\binom{2n}{n-\ell}
 +\sum_{d>0} d^\beta \sum_k \!\binom{2n}{n\!+\!k\!+\!d}\!\binom{2m}{m-k}.
\end{align*}
By Vandermonde's identity, the inner sums over $k$ and $\ell$ are both equal to
$\binom{2m+2n}{m+n+d}.$
Thus, 
\[
T_\beta(m,n) =
2\sum_{d>0} \binom{2m+2n}{m+n+d}d^\beta 
	= \sum_d \binom{2m+2n}{m+n+d} |d|^\beta 
	= S_\beta(m+n).
\]
\end{proof}
\begin{remark}
{\rm
If $m=n$ then, by the shift-invariance of the weight $|k-\ell|^\beta$, we have
\begin{equation}				\label{eq:T_beta_nn}
T_\beta(n,n) = \sum_{k,\;\ell} \binom{2n}{k}\binom{2n}{\ell}|k-\ell|^\beta
	= S_\beta(2n).
\end{equation}
There is no need for the upper argument of the binomial
coefficients to be even in~\eqref{eq:T_beta_nn}. We can adapt the proof
of Theorem~\ref{thm:reduce2to1} to show that, for all $n\in\N$,
\[
\sum_{k,\;\ell} \binom{n}{k}\binom{n}{\ell}|k-\ell|^\beta
        = S_\beta(n).
\]
}
\end{remark}

\section{Centered double sums}	\label{sec:2fold}

In this section we consider the centered double binomial sums defined by%
\footnote{The double sum $S_{\alpha,\beta}(n)$ should not be confused
with the single sum $S_\alpha(n)$ of~\S\ref{sec:reducible}.}

\begin{equation}	\label{eq:Srmndef}
S_{\alpha,\beta}(n) := \sum_{k,\;\ell}\binom{2n}{n+k}\binom{2n}{n+\ell}
  |k^\alpha-\ell^\alpha|^\beta.
\end{equation}
Note that $S_{1,\beta}(n) = T_\beta(n,n)$, so the case $\alpha=1$ is
covered by Theorem~\ref{thm:reduce2to1}. Thus, in the following we
can assume that $\alpha \ge 2$.
Since we mainly consider the case $\alpha=2$, 
it is convenient to define
\begin{equation}			\label{eq:Wrdef}
W_\beta(n) := S_{2,\beta}(n)
 = \sum_{k,\;\ell}\binom{2n}{n+k}\binom{2n}{n+\ell}|k^2-\ell^2|^\beta.
\end{equation}

\begin{remark}
{\rm
The sequences $(S_{\alpha,\beta}(n))_{n\ge 1}$ for $\alpha\in\{1,2\}$ and
$1\le \beta\le 4$ are in the OEIS~\cite{OEIS}.
Specifically, $(S_{1,1}(n))_{n\ge 1}$ is a subsequence of A166337
(the entry corresponding to $n=0$ must be discarded).
$(S_{2,1}(n))_{n\ge 0}$ is A254408,
and $(S_{\alpha,\beta}(n))_{n\ge 0}$ for 
$(\alpha,\beta) = (1,2), (2,2), (1,3), (2,3), (1,4), (2,4)$
are A268147, A268148, \ldots, A268152 respectively.
}
\end{remark}

\subsection{\texorpdfstring{$W_\beta$ for odd $\beta$}{The odd case}}
 \label{subsec:odd}

The analysis of $W_\beta(n)$ 
naturally splits into two cases, depending on the parity of
$\beta$. We first consider the case that $\beta$ is odd.
A simpler approach is possible when $\beta$ is even, as we show
in~\S\ref{subsec:even}.

As mentioned in \S\ref{sec:intro}, the evaluation of $W_1(n)$
was the motivation for this paper, and is given
in the following theorem.

\begin{theorem}[Brent and Osborn]			\label{thm:S12}
\begin{equation*}
W_{1}(n) =\sum_{k,\;\ell}\binom{2n}{n+k}\binom{2n}{n+\ell}|k^2-\ell^2|
 = 2n^2\binom{2n}{n}^2.
\end{equation*}
\end{theorem}
\noindent Numerical evidence suggested the following
generalization of Theorem~\ref{thm:S12}.
It was conjectured by the present authors~\cite[Conjecture~2]{rpb260},
and proved
by Krattenthaler and Schneider~\cite{KS}.
\begin{theorem}[Krattenthaler and Schneider] \label{thm:double_sum_ineq}
For all $m, n\in \N$,
\begin{equation*}
\sum_{k,\;\ell}\binom{2m}{m+k}\binom{2n}{n+\ell}|k^2-\ell^2|
 \ge 2mn\binom{2m}{m}\binom{2n}{n},
\end{equation*}
with equality if and only if $m=n$.
\end{theorem}

\pagebreak[3]

\subsection{Recurrence relations for the odd case}	\label{sec:Grec}

Theorem~\ref{thm:S12} gives 
$W_1(n)$. We show how $W_3(n), W_5(n), \ldots$ can
be computed using recurrence relations.  More precisely, we express the
double sums
$W_{2k+1}(n)$ in terms of certain single sums $G_k(n,m)$, and
give a recurrence for the $G_k(n,m)$.  We then show that
$W_{2k+1}(n)$ is a linear combination of
$P_k(n),\ldots, P_{2k}(n)$, where the polynomials $P_m(n)$ are
as in~\eqref{eq:PQ_def},
and the coefficients
multiplying these polynomials satisfy another recurrence relation.

Define
\[
f_q = 	\begin{cases}
	1 \text{ if } q \ne 0;\\
	\frac{1}{2} \text{ if } q = 0.\\
	\end{cases}
\]
Using symmetry and the definition~\eqref{eq:Wrdef} of $W_k(n)$,
we have
\begin{equation}
W_{2k+1}(n) = 8 \sum_{q=0}^n\sum_{p=q}^n\binom{2n}{n+p}\binom{2n}{n+q}
		(p^2-q^2)^{2k+1}f_q;			\label{eq:Wnoabs}
\end{equation}
the factor $f_q$ allows for terms which would
otherwise be counted twice. 

Let $m=p-q$. Since $p^2-q^2 = m(m+2q)$,  we can write the double sum
$W_{2k+1}(n)/8$ in~\eqref{eq:Wnoabs} as
\begin{equation}					\label{eq:Gksum}
\sum_{q=0}^n\sum_{p=q}^n\binom{2n}{n+p}\binom{2n}{n+q}(p^2-q^2)^{2k+1}f_q
 	= \sum_{m\ge 0}m^{2k+1}G_k(n,m),
\end{equation}
where
\begin{equation}					\label{eq:Gkdef}
G_k(n,m) := \sum_{q\ge 0}\binom{2n}{n+m+q}\binom{2n}{n+q}(m+2q)^{2k+1}f_q.
\end{equation}
Observe that $G_k(0,m) = 0$.
For convenience we define $G_k(-1,m)=0$.
We observe that $G_k(n,m)$ satisfies a recurrence relation, as follows.
\begin{lemma}
For all $k, m, n \in \N$,
\begin{align}
G_{k+2}(n,m) &= 2(4n^2+m^2)G_{k+1}(n,m) - (4n^2-m^2)^2G_k(n,m) \nonumber\\
	&\;\;\;\;	+\,64n^2(2n-1)^2G_k(n-1,m).		\label{eq:G_rec}
\end{align}
\end{lemma}
\begin{proof}
If $n=0$ the proof of~\eqref{eq:G_rec} is trivial,
since $G_k(0,m) = G_k(-1,m) = 0$.
Hence, suppose that $n > 0$. We observe that
\begin{align*}
&[(m+2q)^4-2(4n^2+m^2)(m+2q)^2+(4n^2-m^2)^2]
 \binom{2n}{n+m+q}\binom{2n}{n+q}\\
&= 16(n+m+q)(n-m-q)(n+q)(n-q)
 \binom{2n}{n+m+q}\binom{2n}{n+q}\\
&= 64n^2(2n-1)^2\binom{2n-2}{n-1+m+q}\binom{2n-2}{n-1+q}.
\end{align*}
Now multiply each side by $(m+2q)^{2k+1}f_q$ and sum over $q \ge 0$.
\end{proof}

The recurrence~\eqref{eq:G_rec} may be used to compute
$G_k(n,m)$ for given $(n,m)$ and $k = 0, 1, 2, \ldots$,
using the initial values
\[G_0(n,m) = \frac{n}{2}\,\binom{2n}{n}\binom{2n}{n+m}\]
and
\[G_1(n,m) = \frac{4n^2 + (2n-5)m^2}{2n-1}\,G_0(n,m).\]
These initial values may be verified from the definition~\eqref{eq:Gkdef}
by standard methods~\cite{AeqB}~-- we omit the details.

Write $g_k(n,m) = 0$ if $G_k(n,m) = 0$, and otherwise define $g_k(n,m)$ by 
\[G_k(n,m) = \binom{2n}{n}\binom{2n}{n+m}g_k(n,m).\]
\noindent
The recurrence~\eqref{eq:G_rec} for $G_k$ gives a corresponding
recurrence for $g_k$:
\begin{align}
g_{k+2}(n,m) &= 2(4n^2+m^2)g_{k+1}(n,m) - (4n^2-m^2)^2g_k(n,m) \nonumber\\
	&\;\;\;\;	+ 16n^2(n^2-m^2)g_k(n-1,m), 	\label{eq:g_rec}
\end{align}
with initial values 
\[			
g_0(n,m) = \frac{n}{2}\raisecomma \;\;\;\;
g_1(n,m) = \frac{4n^2 + (2n-5)m^2}{2n-1}\, g_0(n,m).
\]
Note that the $g_k(n,m)$ are rational functions in $n$ and $m$;
if computation with bivariate polynomials over $\Z$ is desired
then $g_k(n,m)$ can  be multiplied by
$(2n-1)(2n-3)\cdots(2n-(2k-1))$. 
If $n$ is fixed, then $g_k(n,m)$ is an even polynomial in $m$ and,
from the recurrence~\eqref{eq:g_rec}, the degree is $2k$.
This suggests that we should
define rational functions $\gamma_{k,j}(n)$ by
\[g_k(n,m) = \sum_{j=0}^{k}\gamma_{k,j}(n)m^{2j}.\]
For $j < 0$ or $j > k$ we define $\gamma_{k,j}(n) = 0$.
{From} the recurrence~\eqref{eq:g_rec}, we obtain the following recurrence
for the $\gamma_{k,j}(n)$:
\begin{align}
\gamma_{k+2,j}(n) &= 8n^2\gamma_{k+1,j}(n) 
	+2\gamma_{k+1,j-1}(n)
	-16n^4\gamma_{k,j}(n)
	+8n^2\gamma_{k,j-1}(n)	\nonumber	\\
&\;\;\;\; -\;\gamma_{k,j-2}(n)		
	+16n^4\gamma_{k,j}(n-1)
	-16n^2\gamma_{k,j-1}(n-1).		\label{eq:alpharec}
\end{align}
The $\gamma_{k,j}(n)$ can
be computed from~\eqref{eq:alpharec}, using the initial values
\begin{align}
\gamma_{0,0}(n) &= n/2,\nonumber\\
\gamma_{1,0}(n) &= 2n^3/(2n-1),\label{eq:gamma_init}\\
\gamma_{1,1}(n) &= n(2n-5)/(4n-2).\nonumber
\end{align}
Using the definition of $\gamma_{k,j}(n)$
and~\eqref{eq:Wnoabs}--\eqref{eq:Gkdef}, we obtain
\[W_{2k+1}(n) = 4\binom{2n}{n}\sum_{j=0}^{k}\gamma_{k,j}(n)S_{2k+2j+1}(n).\]
Since $S_{2r+1}(n) = P_r(n)n\binom{2n}{n}$,
we obtain the following theorem, which shows that the 
double sums $W_{2k+1}(n)$
may be expressed in terms of the same polynomials $P_{m}(n)$ that
occur in expressions for the single sums of~\cite{rpb259,Tuenter1}.
\begin{theorem}			\label{thm:Wodd}
\begin{equation}		\label{eq:Woddsum}
W_{2k+1}(n) = 4n\sum_{j=0}^k\gamma_{k,j}(n)P_{k+j}(n) \cdot
	\binom{2n}{n}^2,
\end{equation}
where the polynomials $P_{k+j}(n)$ are as in~\eqref{eq:PQ_def},  
and the $\gamma_{k,j}(n)$ may be computed from the
recurrence~\eqref{eq:alpharec}
and the initial values given in~\eqref{eq:gamma_init}.
\end{theorem}
The factor before the binomial coefficient in~\eqref{eq:Woddsum}
is a rational function 
$\omega_k(n)$ with
denominator $(2n-1)(2n-3)\cdots(2n- 2\lceil k/2 \rceil +1)$.
Thus, we have the following corollary of Theorem~\ref{thm:Wodd}.
\begin{corollary}			\label{cor:Wodd}
If $k\in \N$ and  $W_k(n)$ is defined by~\eqref{eq:Wrdef}, then
\[W_{2k+1}(n) = \omega_k(n)\binom{2n}{n}^2,\]
where
\vspace*{-10pt} 
\[\omega_k(n)\prod_{j=1}^{\lceil k/2 \rceil}(2n-2j+1)\]
is a polynomial of degree $2k+\lceil k/2\rceil + 2$ over~$\Z$.
The first four cases are:
\begin{align*}
\omega_0(n) &= 2n^2,\\ 
\omega_1(n) &= \frac{2n^3(8n^2-12n+5)}{2n-1}\,\raisecomma\\
\omega_2(n) &=\frac{2n^3(128n^4-512n^3+800n^2-568n+153)}
 {2n-1}\,\raisecomma\;\;\text{and}\\
\omega_3(n) &= \frac{2n^3\,\overline{\omega}_3(n)}
 {(2n-1)(2n-3)}\,\raisecomma\;\;\text{where}\\
 \overline{\omega}_3(n) &=
 9216n^7-86016n^6+350464n^5-802304n^4+\\
 &\phantom{=\;\;} 1106856n^3-914728n^2+417358n-80847.
\end{align*}
\end{corollary}

\subsection{
\texorpdfstring{$W_\beta$ for even $\beta$}{The even case}}
							\label{subsec:even}
Now we consider $W_\beta(n)$ for even~$\beta$.
This case is easier than the case
of odd~$\beta$ 
because the absolute value
in the definition~\eqref{eq:Wrdef} has no effect when $\beta$ is even.
Theorem~\ref{thm:Weven} shows that $W_{2r}(n)$ can be expressed in terms
of the single sums $S_0(n), S_2(n), \ldots, S_{4r}(n)$
or, equivalently, in terms of the polynomials
$Q_0(n), Q_1(n), \ldots, Q_{2r}(n)$.
It follows 
that $2^{2r-4n}W_{2r}(n)$ is a polynomial over $\Z$ of degree $2r$ in~$n$.

\pagebreak[3]
\begin{theorem}		\label{thm:Weven}
For all $n\in\N$,
\begin{align*}
W_{2r}(n) &= \sum_k (-1)^k\binom{2r}{k}S_{2k}(n)S_{4r-2k}(n)\\
 &= 2^{4n-2r} \sum_k (-1)^k\binom{2r}{k}Q_{k}(n)Q_{2r-k}(n),
\end{align*}
where $Q_r(n)$ and $S_r(n)$ are as
\eqref{eq:Sr_def}--\eqref{eq:PQ_def} of~$\S\ref{sec:reducible}$,
and
$W_\beta(n)$ is defined by $\eqref{eq:Wrdef}$.
\end{theorem}
\begin{proof}
{From} the definition of $W_{2r}(n)$ we have
\[
W_{2r}(n) = \sum_i \sum_j \binom{2n}{n+i}\binom{2n}{n+j}(i^2-j^2)^{2r}.
\]
Write
\[(i^2-j^2)^{2r} = \sum_k (-1)^k\binom{2r}{k}i^{4r-2k}j^{2k},\]
change the order of summation in the resulting triple sum, and observe
that the inner sums over $i$ and $j$ separate, giving
$S_{4r-2k}(n)S_{2k}(n)$.
This proves the first part of the theorem.  The second part follows
from~\eqref{eq:PQ_def}.
\end{proof}
For example, the first four cases are
\begin{align*}
W_0(n) &= 2^{4n},\\
W_2(n) &= 2^{4n-1}\,n(2n-1),\\
W_4(n) &= 2^{4n-2}\,n(2n-1)(18n^2-33n+17),\\
W_6(n) &= 2^{4n-3}\,n(2n-1)(900n^4-4500n^3+8895n^2-8055n+2764).
\end{align*}
It follows from Theorem~\ref{thm:Weven}
that the coefficients of $2^{2r-4n}W_{2r}(n)$ are in~$\Z$,
but it is not obvious how to prove the stronger result,
suggested by the cases above,
that the coefficients of $2^{r-4n}W_{2r}(n)$ are in $\Z$.
We leave this as a conjecture.

\pagebreak[3]

\section{A triple sum}			\label{sec:triple}

In Theorem~\ref{thm:3-fold}
we give a triple sum that is analogous to the double sum of
Theorem~\ref{thm:S12}.
A straightforward but tedious proof is given in~\cite[Appendix]{rpb260}.
The result also follows from the case $d=3$ of 
a more general result proved in~\cite[Proposition 1.1]{rpb263}
for the analogous
$d$-fold sum,
where the weight function
is generalized to the absolute value of a 
Vandermonde $|\Delta(i_1^2,i_2^2,\ldots,i_d^2)|$.

\begin{theorem}		\label{thm:3-fold}
For all $n\in\N$,
\begin{align*}		
\sum_{i,\;j,\;k}&\binom{2n}{n+i}\binom{2n}{n+j}\binom{2n}{n+k}
 |(i^2-j^2)(i^2-k^2)(j^2-k^2)|\\ 
 &= \; 3n^3(n-1)\binom{2n}{n}^2 2^{2n-1}.
\end{align*}
\end{theorem}

\section{Further identities} \label{sec:further}

In this section we give various identities that were stated in~\cite{rpb260}.
Of these, \eqref{eq:S13}, \eqref{eq:S15}, \eqref{eq:S17},
\eqref{eq:O3} and \eqref{eq:O5} were conjectural.
The conjectures have since been proved by
Bostan, Lairez and Salvy~\cite[\S7.3.2]{Bostan15}.

\subsection*{Centered double sums}

Recall that, from the definition~\eqref{eq:Srmndef}, we have
\begin{equation}	\label{eq:S1mndef}
S_{\alpha,1}(n) = \sum_{i,\;j}\binom{2n}{n+i}\binom{2n}{n+j}
 |i^\alpha-j^\alpha|.
\end{equation}
We give closed-form expressions for
$S_{\alpha,1}(n)$, $1 \le \alpha \le 8$.
Observe that~\eqref{eq:S11} follows from Theorem~\ref{thm:reduce2to1}
since $S_{1,1}(n) = T_1(n,n)$,
and~\eqref{eq:S12} is equivalent to Theorem~\ref{thm:S12}.
It appears that, for even $\alpha$, $S_{\alpha,1}(n)$ 
is a rational function of $n$
multiplied by $\binom{2n}{n}^2$,
but for odd $\alpha$, it is a rational function
of~$n$ multiplied by $\binom{4n}{2n}$.
This was conjectured in~\cite{rpb260}, and has been proved
by Krattenthaler and Schneider~\cite{KS}.

\begin{align}	
S_{2,1}(n) &= 2n^2\binom{2n}{n}^2,	\label{eq:S12}\\
S_{4,1}(n) &=\; \frac{2n^3(4n-3)}{2n-1}
 \binom{2n}{n}^2,	\label{eq:S14}\\ 
S_{6,1}(n) &=\; \frac{2n^3(11n^2-15n+5)}{2n-1}
 \binom{2n}{n}^2, \label{eq:S16}\\ 
S_{8,1}(n) &=\; \frac{2n^3(80n^4-306n^3+428n^2-266n+63)}
 {(2n-1)(2n-3)} \binom{2n}{n}^2, \label{eq:S18} 
\\
S_{1,1}(n) &= 2n\binom{4n}{2n},		\label{eq:S11}\\
S_{3,1}(n) &= \frac{4n^2(5n-2)}{4n-1}\binom{4n-1}{2n-1},
\label{eq:S13} 
\end{align}
\begin{align}
S_{5,1}(n) &= \frac{8n^2(43n^3-70n^2+36n-6)}{(4n-2)(4n-3)}
 \binom{4n-2}{2n-2},
  \label{eq:S15}\\ 
S_{7,1}(n) &= \frac{16n^2 P_{7,1}(n)}
 {(4n-3)(4n-4)(4n-5)} \binom{4n-3}{2n-3},\;\; n \ge 2, \text{ where} \nonumber\\
 P_{7,1}(n) &= 531n^5-1960n^4+2800n^3-1952n^2+668n-90,
 \label{eq:S17}\\		
 (S_{7,1}(1) &= 12 \text{ is a special case).} \nonumber
\end{align}

Following are some similar identities.
We observe that,
since $i^4-j^4 = (i^2+j^2)(i^2-j^2)$,
~\eqref{eq:O1} is easily seen to be equivalent to~\eqref{eq:S14}.
Similarly, since $i^6-j^6 = (i^4 + i^2j^2 + j^4)(i^2-j^2)$,
any two of \eqref{eq:S16}, \eqref{eq:O2} and \eqref{eq:O4} imply the third.
Higher-dimensional generalizations of~\eqref{eq:O3}--\eqref{eq:O4} are
known~\cite{rpb263}. 

\begin{align}	
\sum_{i,\;j}\binom{2n}{n\!+\!i}\binom{2n}{n\!+\!j} |i^2(i^2-j^2)|
 &=\;\frac{n^3(4n-3)}{2n-1}\,\binom{2n}{n}^2, \label{eq:O1}\\
\sum_{i,\;j}\binom{2n}{n\!+\!i}\binom{2n}{n\!+\!j} |i^4(i^2-j^2)|
 &=\; \frac{n^3(10n^2-14n+5)}{2n-1}\,\binom{2n}{n}^2, \label{eq:O2}\\
\sum_{i,\;j}\binom{2n}{n\!+\!i}\binom{2n}{n\!+\!j} |ij(i^2-j^2)|
 &= \frac{2n^3(n-1)}{2n-1}\,\binom{2n}{n}^2, \label{eq:O3}\\
\sum_{i,\;j}\binom{2n}{n\!+\!i}\binom{2n}{n\!+\!j} |i^2 j^2(i^2-j^2)|
 &=\; \frac{2n^4(n-1)}{2n-1}\,\binom{2n}{n}^2, \label{eq:O4}\\
\!\!\sum_{i,\;j}\binom{2n}{n\!+\!i}\binom{2n}{n\!+\!j} |i^3 j^3(i^2-j^2)|
 &=
  \frac{2n^4(n-1)(3n^2-6n+2)}{(2n-1)(2n-3)}\,\binom{2n}{n}^2\!.
   \label{eq:O5}
\end{align}

\section{Acknowledgements}
We thank Christian Krattenthaler for sending us
proofs~\cite{KS} of some of the conjectures that
we made in an earlier version~\cite{rpb260} of this paper,
and informing us of the work of Bostan \emph{et al}~\cite{Bostan15}.
We also thank David Callan for sending us his short
proof~\cite{Callan} of~\eqref{eq:BO_rpb255} using lattice paths
(specifically, equations (1.12) and (1.14) of~\cite{GKS}).
The first author was supported in part
by Australian Research Council grant DP140101417.

\pagebreak[3]

\bigskip
\hrule
\bigskip

\noindent 2010 {\it Mathematics Subject Classification}:
Primary 05A10, 11B65;
Secondary 05A15, 05A19, 44A60, 60G50.

\noindent \emph{Keywords: } 
absolute binomial sum, binomial sum, centered binomial sum,
Dumont-Foata polynomial, Tuenter polynomial, Vandermonde kernel.

\bigskip
\hrule
\bigskip

\noindent (Concerned with sequences
\seqnum{A166337},
\seqnum{A254408},
\seqnum{A268147},
\seqnum{A268148},
\seqnum{A268149},\linebreak 
\seqnum{A268150},
\seqnum{A268151},
\seqnum{A268152}.) 

\bigskip
\hrule
\bigskip

\vspace*{+.1in}
\noindent

\bigskip
\hrule
\bigskip

\noindent


\begin{thebibliography}{99}
 


\bibitem{Best}
M. R. Best, The excess of a Hadamard matrix,
\emph{Nederl.\ Akad.\ Wetensch.\ Proc.\ Ser. A} \textbf{80}
$=$ 
\emph{Indag.\ Math.\ }\textbf{39} (1977), 
357--361.


\bibitem{Bostan15}
A. Bostan, P. Lairez and B. Salvy,
\emph{Multiple binomial sums},
arXiv:\linebreak 
1510.07487v1, 26 Oct.\ 2015.

\bibitem{rpb259}
R. P. Brent,
Generalising Tuenter's binomial sums,
\emph{Journal of Integer Sequences} 18 (2015), article 15.3.2, 9 pp.

\bibitem{rpb263}
R. P. Brent, C. Krattenthaler and S. O. Warnaar,
\emph{Discrete analogues of Mehta-type integrals},
arXiv:1601.06536v1, 25 Jan.\ 2016.

\bibitem{rpb260} 
R. P. Brent,
H. Ohtsuka, 
J. H. Osborn,
and H. Prodinger,
\emph{Some binomial sums involving absolute values},
arXiv:1411.1477v1, 6 Nov.\ 2014.

\bibitem{rpb255}
R. P. Brent and J. H. Osborn,
\emph{Note on a double binomial sum relevant to the Hadamard maximal
determinant problem}, arXiv:1309.2795v2, 12 Sept.\ 2013.


\bibitem{rpb257}
R. P. Brent, J. H. Osborn and W. D. Smith,
\emph{Lower bounds on maximal determinants of binary matrices via
the probabilistic method}, arXiv:1402.6817v2, 13 March 2014.



\bibitem{Callan} 
D. Callan,
Remark on ``Some binomial sums involving absolute values'',
personal communication to H. Prodinger, 17 Nov.\ 2014.



\bibitem{DF}
D. Dumont and D. Foata, 
Une propri\'et\'e de sym\'etrie des nombres de Genocchi,
\emph{Bulletin de la Soci\'et\'e Math\'ematique de France}
\textbf{104} (1976), 433--451.



\bibitem{Gosper}
R. W. Gosper, Jr., 
Decision procedure for indefinite hypergeometric summation,
\emph{Proc.\ Natl.\ Acad.\ Sci.\ USA} \textbf{75}, 1 (1978), 40--42.

\bibitem{Gould}
H. W. Gould,	
\emph{Combinatorial Identities},
2nd edition, Morgantown, West Virginia, USA, 1972. 

\bibitem{GKP}
R. L. Graham, D. E. Knuth and O. Patashnik,
\emph{Concrete Mathematics},
Addison-Wesley, Reading, Massachusetts, 1989.


\bibitem{GKS}
R. K. Guy, C. Krattenthaler and B. E. Sagan,
Lattice paths, reflections, and dimension-changing bijections,
\emph{Ars Combinatorica} 34 (1992), 3--15.


\bibitem{KS}
C. Krattenthaler and C. Schneider,
Evaluation of binomial double sums involving absolute values,
preprint, 2015.




\bibitem{AeqB}
M. Petkov\v{s}ek, H. S. Wilf and D. Zeilberger,
\emph{$A = B$}, A.~K.~Peters Ltd, Wellesley, Mass., 1996.

\bibitem{Prodinger1}
H. Prodinger,	
\emph{A short and elementary proof for a double sum of Brent
and Osburn}, 
arXiv:1309.4351v1, 17 Sept.\ 2013, 2~pp.


\bibitem{OEIS}	
N. J. A. Sloane \emph{et al},
The On-Line Encyclopedia of Integer Sequences,
\url{https://oeis.org/}.


\bibitem{Tuenter1}
H. J. H. Tuenter, 
Walking into an absolute sum,
\emph{The Fibonacci Quarterly} \textbf{40} (2002), 175--180.
Also arXiv:math/0606080v1, 4 June 2006.


\bibitem{Wilf}
H. S. Wilf,	
\emph{generatingfunctionology},
3rd edition, A.~K.~Peters Ltd, Wellesley, Mass., USA, 2006.
 
\end{thebibliography}
\end{document}